\documentclass{amsart}
\RequirePackage{float}%图片浮动
\RequirePackage[margin=1in,dvips]{geometry}
\RequirePackage{graphicx}% 插图包
\RequirePackage{amsmath, amsthm, amssymb, amsfonts,slashed, pdfsync}%数学环境
\RequirePackage[colorlinks=true]{hyperref}%交叉引用
\hypersetup{urlcolor=blue, citecolor=blue}%链接颜色
\RequirePackage[nameinlink]{cleveref}%智能引用
\RequirePackage{xifthen}%判断逻辑
\RequirePackage{subcaption}
\RequirePackage{indentfirst}
\RequirePackage{microtype}%文字间距优化(和中文一起)
\newenvironment{Eq}[1][]{\begin{equation}\ifthenelse{\equal{#1}{}}{}{\tag{#1}}\begin{aligned}}{\end{aligned}\end{equation}\ignorespacesafterend}
\newenvironment{Eq*}[1][]{\begin{equation*}\ifthenelse{\equal{#1}{}}{}{\tag{#1}}\begin{aligned}}{\end{aligned}\end{equation*}\ignorespacesafterend}
\numberwithin{equation}{section}
\newtheorem{theorem}{Theorem}[section]
\newtheorem{proposition}[theorem]{Proposition}

\newtheorem{lemma}[theorem]{Lemma}

\theoremstyle{remark}%
\newtheorem{remark}{Remark}[section] 
\theoremstyle{definition}
\DeclareMathOperator{\fd}{d}	\renewcommand{\d}{\fd}

\newcommand{\const}{\mathrm{const}.}
\newcommand{\Roma}[1]{\uppercase\expandafter{\romannumeral#1}}
\newcommand{\Alphab}{{\underline{\mathcal A}}}
\newcommand{\varC}{{\mathcal C}}
\newcommand{\alphab}{{\underline\alpha}}
\DeclareMathOperator{\I}{i}

\newcommand{\mm}{{\mathbf{m}}}

\newcommand{\sA}{{\slashed A}}
\def\pa{\partial}

\newcommand{\rR}{{\mathbb{R}}}
\newcommand{\rS}{{\mathbb{S}}}

\newcommand{\rD}{{\mathcal D}}
\newcommand{\rH}{{\mathcal H}}
\newcommand{\rHb}{{\underline{\mathcal H}}}

\newcommand{\Lb}{{\underline L}}
\newcommand{\hL}{{\hat L}}
\newcommand{\hLb}{{\hat \Lb}}
\newcommand{\spartial}{{\slashed \partial}}

\newcommand{\kl}[1]{\mathopen{}\left#1}
\newcommand{\kr}[1]{\right#1}
\newcommand{\Low}[1]{\limits_{\substack{#1}}}
\newcommand{\low}[1]{_{\substack{#1}}}

\crefname{equation}{}{}%配套于cleverref
\newcommand{\Tm}[1]{Theorem \ref{#1}}

\newcommand{\La}[1]{Lemma \ref{#1}}

\newcommand{\Pn}[1]{Proposition \ref{#1}}

\renewcommand{\tt}{{\tilde t}}
\newcommand{\tx}{{\tilde x}}
\newcommand{\tr}{{\tilde r}}
\newcommand{\tomega}{{\tilde\omega}}
\newcommand{\tu}{{\tilde u}}
\newcommand{\tv}{{\tilde v}}
\newcommand{\tLb}{{\tilde\Lb}}
\newcommand{\tL}{{\tilde L}}
\newcommand{\trS}{{\tilde\rS}}

\newcommand{\trD}{{\tilde\rD}}
\newcommand{\tA}{{\tilde A}}

\newcommand{\tF}{{\tilde F}}
\newcommand{\tphi}{{\tilde \phi}}
\newcommand{\te}{{\tilde e}}

\newcommand{\tD}{{\tilde D}}
\newcommand{\talpha}{{\tilde \alpha}}
\newcommand{\tpartial}{{\tilde \partial}}

\newcounter{part0}[subsection]
\renewcommand{\part}[1][]{\noindent\refstepcounter{part0}{\bfseries Part \number\value{part0}:} #1.\par}

\newcounter{part1}[part0]

\newcounter{part2}[part1]

\title{Characteristic initial value problem for nonlinear wave equation with singular initial data}

\author{Wei Dai}
\address{School of Mathematical Sciences\\ Zhejiang University of Technology\\ Zhejiang, China}
\email{daiw23@zjut.edu.cn}

\author{Shiwu Yang}
\address{Beijing International Center for Mathematical Research\\ Peking University\\ Beijing, China}
\email{shiwuyang@math.pku.edu.cn}

\date{\today}
\begin{document}

\bibliographystyle{plain}

\begin{abstract}
In this paper, we study the characteristic initial value problem for a class of nonlinear wave equations with data on a conic light cone in the Minkowski space $\mathbb{R}^{1+3}$. We show the existence of local solution for a class of singular initial data in the sense that the standard energy could be infinite and the solution may blow up at the conic point. As an application, we improve our previous result on the inverse scattering problem for the Maxwell-Klein-Gordon equations with scattering data on the future null infinity. 
\end{abstract}



\maketitle
%\tableofcontents

\section{Introduction}
\label{Sn:I}
%Let $(t, r, \omega)$ be the standard polar coordinates in the Minkowski space $\mathbb{R}^{1+3}$. 
Let $(t, x)$ be the standard coordinates in the Minkowski space $\mathbb{R}^{1+3}$. 
Introduce the null coordinates 
\[
u=\frac{t-|x|}{2},\quad v=\frac{t+|x|}{2},\quad \omega =\frac{x}{|x|}, \quad r=|x|
\]
 and let  $\mathcal{H}_0=\{u=0,\quad 0\leq v\leq 1\}$ be the conic light cone. This paper is devoted to studying the characteristic initial value problem for a class of semilinear wave equation 
\begin{Eq}
\label{Eq:Me}%Main equation
\square\phi=(-\partial_t^2+\Delta)\phi=F(\phi,\partial\phi),\qquad
\phi|_{\mathcal{H}_0}=\varphi_0(v,\omega)
\end{Eq}
with singular initial data in the sense that $\varphi_0$ may blow up at the conic point $(0, 0)$ of the light cone $\mathcal{H}_0$. Here 
  the nonlinearity $F$ is at least quadratic in terms of $\phi$ and $\pa \phi$.  

Motivated by the studying of some physical models like Einstein equations in mathematical general relativity, the characteristic initial value problem for nonlinear wave equation is of particular importance both from the mathematical and physical point of view. We refer to for example \cite{Chruciel16:CIV:NLW} for a more comprehensive introduction on this subject. The local well-posedness of characteristic initial value problem for nonlinear hyperbolic system has been studied extensively in the past decades, dating back to the classical results of 
 Cagnac in \cite{cagnac81:NW:CI} and 
Rendall in \cite{Rendall90:CIVP:NLW} with sufficiently smooth initial data on the characteristic initial null cone or two intersecting null hypersurfaces. 
Generalizations and applications of these results were contributed in 
\cite{Hagen90:CIVP:NLW},
  \cite{Cagnac94:CIV:NLW},
  \cite{Dossa01:CIV:NLW},
     \cite{Dossa03:CIV:YMH}, 
   \cite{Dossa05:CIV:SLW},
      \cite{Cabet08:CIV:NLW},
   \cite{Dossa10:CIVP:EYMH},
\cite{Friedrich13:EE:CIVP}, 
\cite{Chrusciel13:EE:pastnull},
 \cite{Dossa19:CIVP:NW} 
 \cite{Bap21:CIV:EV} and references therein. A common future of these works, as far as we know, is that the initial data are sufficiently regular at the conic point or the intersecting two sphere. A natural question is whether the initial data can be singular at the conic point so that one still can construct a local solution.

Before stating our main result and to see how singular the data could be, let's take the following initial data and its corresponding 
linear solution  
\begin{Eq*}
\varphi_0=v^{\delta-1},\qquad \phi_0=\frac{\kl(t+r\kr)^\delta-(t-r)^\delta}{2^\delta r}
\end{Eq*}
for example. For fixed $0<\delta\ll 1$, 
it can be checked that $\phi_0$ is smooth and verifies the linear wave equation
\[
\Box\phi_0=0
\]
%in $\mathbb{R}^{1+3}$
for the region above $\mathcal{H}_0$. 
% the light cone $\mathcal{H}_0$. 
Obviously $\varphi_0$ 
is singular at the conic point and has infinite standard energy. 
Our first result is to show that the solution to the nonlinear wave equation \eqref{Eq:Me} exists in a  neighborhood of the light cone $\mathcal{H}_0$ for such singular initial data. Without loss of generality and for simplicity, we will investigate the following type of nonlinearity
\begin{Eq}\label{Eq:toF}
F(\phi,\partial\phi)=C_0\phi^3+ C_1^\mu\phi\partial_\mu\phi
\end{Eq}
with fixed constants $C_0$ and  $C_1^\mu$. Commutators used in this paper consist of the scaling vector field and the Lorentz rotations 
 \[
 \Omega=\{\Omega_{ij}=x_i\partial_j-x_j\partial_i\},\quad Z=\{S=t\partial_t+r\partial_r,\quad L_i=t\partial_i+x_i\partial_t, \quad \Omega_{ij}\}. 
 \]
 We may also use $\spartial $ to be short for the angular derivative $r^{-1}\pa_{\omega}$. 
\begin{theorem}\label{Tm:M1}%Main 1
Consider the characteristic initial value problem to the semilinear wave equation \ref{Eq:Me} with \eqref{Eq:toF}. 
Assume that the initial data $\varphi_0$ on the initial cone $\mathcal{H}_0$ is bounded in the following sense
\begin{Eq}\label{Eq:Cod}%Condition of data
\sum\Low{n_1\leq 2\\n_2\leq 5-n_1}\int_0^1 \int_{\rS^2}v^{1-2\delta}\kl|(v\partial_v)^{n_1}\partial_\omega^{n_2}\varphi_0\kr|^2\d\omega\d v \leq M_0^2
\end{Eq}
for some positive constant $M_0$ and a small positive constant $\delta$. 
Then there exist positive constants $M_0'$, $\varepsilon$, depending only on $M_0$ and $\delta $, such that the characteristic initial value problem to \eqref{Eq:Me} admits a unique solution $\phi$  in the neighborhood of light cone $ \{ 0\leq v\leq 1,0\leq u\leq \varepsilon\}$ verifying the following weighted energy estimates
\begin{Eq}\label{Eq:Er}%Energy result
&\sum\limits_{k\leq 1, l\leq 2}\sup_{0\leq u\leq U}\int_u^V\int_{\rS^2}v\kl|\partial_v(rZ^k\Omega^l\phi )\kr|^2+u\kl|\spartial(rZ^k\Omega^l\phi )\kr|^2\d\omega\d v\\
 +&\sum\limits_{k\leq 1, l\leq 2}\sup_{0\leq v\leq V}\int_0^{\min\{v,U\}}\int_{\rS^2}u\kl|\partial_u(rZ^k\Omega^l\phi )\kr|^2+v\kl|\spartial(rZ^k\Omega^l\phi )\kr|^2\d\omega\d u\leq {M_0'}^2V^{2\delta}
\end{Eq}
for any $0\leq U\leq \min\{\varepsilon,V\}\leq 1$. Consequently the solution also satisfies the pointwise bound 
 $$|\phi(U,V,\omega)|\leq M_0' V^{\delta-1},\quad \forall 0 \leq V\leq 1, \quad 0 \leq U\leq \varepsilon. $$ 
\end{theorem}
\begin{remark}
Assuming $\phi \sim v^{-\alpha}$, we then formally have $\partial^k \phi \sim v^{-\alpha-k}$. This implies that the singularity of the nonlinear term is at most $v^{-\alpha-2}$.  
Thus, nonlinear terms such as $(\partial \phi)^2$ cannot appear, while nonlinear terms like $\phi^4$ would require a bit lower singularity.
\end{remark}
This local existence result for the characteristic initial value problem with singular initial data also holds for much more general nonlinearity. The above particular nonlinear model is partially motivated by the inverse scattering problem for the Maxwell-Klein-Gordon system studied by the authors in \cite{Yang:MKG:scattering}. Let $A$ be a 1-form in the Minkowski space $\mathbb{R}^{1+3}$. For complex scalar field $\phi$, the associated covariant derivative is given by 
\begin{Eq*}
D \phi =\partial \phi+\I A \phi ,\qquad \I:=\sqrt{-1}. 
\end{Eq*}
Recall the Maxwell-Klein-Gordon  system (MKG) for the Maxwell field $F=dA$ and the scalar field $\phi $
\begin{Eq}
\label{Eq:MKGe-Fphi}%MKG equation  - Fphi
\begin{cases}
\partial^\mu F_{\mu\nu}=-\Im(\phi\cdot\overline{D_\nu \phi}),\\
\square_A\phi:=D^\mu D_\mu \phi=0.
\end{cases}
\end{Eq}
In Minkowski space $\mathbb{R}^{1+3}$,  introduce a null frame $\{\Lb,L,e_1,e_2\}$ adapted to the null coordinates $\{u, v, \omega\}$, in which 
\[
\Lb=\partial_u=\pa_t-\pa_r,\quad L=\pa_v=\pa_t+\pa_r
\] 
 and ${e_1, e_2}$ is an orthonormal basis of the tangent space of the two-sphere with radius $r$. Let 
 \[
\alphab_j=F(\Lb,  e_j),\quad j=1, 2. 
 \]
 For fixed $u\in \mathbb{R}$ and $\omega\in \mathbb{S}^2$, the scattering data $(\Alphab,\Phi)$ at the future null infinity is given by 
 \begin{Eq}
\label{eq:radiation}
\Alphab(u, \omega)=\lim_{v\rightarrow\infty} (r\alphab),\qquad
\Phi(u, \omega)=\lim_{v\rightarrow\infty} (r\phi).
\end{Eq} 
It is known from the work \cite{YangYu:MKG:smooth} that the limit exists for the Cauchy problem. The inverse scattering problem is to construct solutions to the MKG equation with given scattering data  $\Alphab$ and $\Phi$ on the future null infinity. 

For any $U_*\in\rR$ recall the weighted energy $\|\cdot\|_{S\!N_{U_*}}$ for the scattering data
\begin{Eq*}
\|(\Alphab,\Phi)\|_{S\!N_{U_*}}^2:=&\hspace{-5pt}\sum\Low{n_1\leq 1\\n_2\leq 6-3n_1}\hspace{-5pt}\int_{U_*}^\infty \int_{\mathbb{S}^2} \kl<u\kr>^{1+2\delta+2n_1}|\partial_u^{n_1} \partial_\omega^{n_2}\Alphab|^2\d\omega\d u\\
&\quad+\sum\Low{n_1\leq 2\\n_2\leq 5-2\max\{n_1-1,0\}}\int_{U_*}^\infty \int_{\mathbb{S}^2} \kl<u\kr>^{-1+2\delta+2n_1}|D_u^{n_1} D_\omega^{n_2}\Phi|^2\d\omega\d u
\end{Eq*} 
defined in \cite{Yang:MKG:scattering}. Here $\kl<u\kr>:=\sqrt{|u|^2+1}$ and the covariant derivative $D$ can be  determined by $\Alphab$.
See  \cite{Yang:MKG:scattering} for a more detailed discussions.
Based on the  local existence result of Theorem \ref{Tm:M1}, we can show that 
\begin{theorem}\label{Tm:M2}
For any given scattering data $(\Alphab,\Phi)$ such that  
 $ \|(\Alphab,\Phi)\|_{S\!N_{U_*}}$ is finite for any $U_*\in \mathbb{R}$, there exists a solution $(F,\phi)$ to the MKG system \eqref{Eq:MKGe-Fphi} which scatters to $(\Alphab,\Phi)$ in the sense of \eqref{eq:radiation}. Moreover the solution is unique if for any $U_*\in \mathbb{R}$
\begin{Eq}
\label{eq:uniqueness}
\lim_{v\rightarrow\infty}\|(r\alphab,r\phi)\|_{S\!N_{U_*}'}=\|(\Alphab,\Phi)\|_{S\!N_{U_*}},
\end{Eq}
in which the weighted energy $\|(r\alphab,r\phi)\|_{S\!N_{U_*}'}$ is similarly defined on the light cone  $\{v=\const\}$. 
\end{theorem}
The regularity assumption  in \cite{Yang:MKG:scattering} is that the scattering data decay sufficiently fast at the time like infinity, that is $1+2\delta>6$ in the weighted energy norm, which is based on the local existence result for the characteristic initial value problem with smooth initial data.
However on the other hand, the Cauchy problem studied in \cite{YangYu:MKGrough} indicates that the solution to the MKG equation exists globally and decays like linear wave as long as $\delta>0$. This means that the above theorem improves the regularity for the inverse scattering problem to the case when $\delta>0$, which is then consistent with the Cauchy problem.

The proof for Theorem \ref{Tm:M1} is procedurally 
standard. 
%Since the data is singular at the conic point, we use iteration process to construct the local solution in a neighborhood of the conic point. %没找到逻辑关系
We use iteration process to construct the local solution in a neighborhood of the $\rH_0$. 
The existence of linear solution to such characteristic initial value problem has been discussed for example in \cite{Friedlander75:wave:curved}, which is based on the case when the solution is spherically symmetric. The general case then follows by using Lorentz rotation. See more details in the beginning of Section \ref{sec:Th1}. 

To show that the sequence of linear solutions to the characteristic initial value problem converges, we then need to show that the linear solution lies in certain weighted Sobolev space. The weights, which are polynomial power of the distance to the conic point, captured the possible blow up of the solution near the origin. Partially inspired by the works \cite{Luk12:EE:local}, \cite{Igor:impulsewave:local}, \cite{Igor17:Interaction:2}, such weighted energy estimates are carried out by using vector field method adapted to regions bounded by null hypersurfaces, with the key new multiplier
$$2v\partial_v+u\partial_u,
$$
which is built upon the scaling vector field and the tangential vector field of the out going null hypersurface. The role played by this multiplier is that the bulk integral has the same sign with the weighted energy flux through the future null hypersurfaces, which enables us to control the weighted energy in terms of the initial weighted energy.

\textbf{Acknowledgments.}  
The authors would like to thank the anonymous referee for the careful reading and valuable comments. 
The first author is supported by NSFC (No.TBD).
The second auther is supported by the National Key R\&D Program of China 2021YFA1001700 and the National Science Foundation of China 12171011,  12141102.

\def\pa{\partial}
\section{Preliminary}
\subsection{Additional notations}
To avoid too many constant, throughout this paper, we use the convention that $x\lesssim y$ ( or equivalently $y\gtrsim x$) means $x\leq C y$ ( or equivalently $Cy\geq x$) for some constant $C>0$ 
depending only on $\delta$.

At any fixed point $(t, x)$, we may choose the frame $e_1$ and  $e_2$ such that
\begin{equation*}
 [L, e_j]=-\frac{1}{r}e_j, 
 \quad [\Lb, e_j]=\frac{1}{r}e_j,\quad
[e_1, e_2]=0,\quad j\in\{1, 2\}.
\end{equation*}
under which the covariant derivatives verify
\begin{equation*}
 \begin{split}
&\nabla_L L=0,\quad   \nabla_L \Lb=0, \quad \nabla_L e_j=0,\quad \nabla_{\Lb}\Lb=0, \quad \nabla_{\Lb}e_j=0,\\
&\nabla_{e_j}L=r^{-1}e_j,\quad \nabla_{e_j}\Lb=-r^{-1}e_j, \quad \nabla_{e_1}e_2=\nabla_{e_2}e_1=0,\quad \nabla_{e_j}e_j=-r^{-1}\pa_r.
 \end{split}
\end{equation*}
Here $\nabla$ is the covariant derivative in Minkowski space.

Since we are studying the characteristic initial value problem, outgoing and incoming null hypersurfaces will be of particular interest
\begin{Eq*}
\rS_{U}^V:=&\{(t,x):u=U,~v=V\},\quad&
\rH_{U}^{V_1,V_2}:=&\{(t,x):u=U,~V_1\leq v\leq V_2\},\\
\rHb_{U_1,U_2}^{V}:=&\{(t,x):U_1\leq u\leq U_2,~v=V\},\quad&
\rD_{U_1,U_2}^{V_1,V_2}:=&\{(t,x):U_1\leq u\leq U_2,~V_1\leq v\leq V_2\}.
\end{Eq*}
See the following  figure  for an illustration of the surfaces and regions. 
\begin{figure}[H]
\centering
\includegraphics[width=0.3\textwidth]{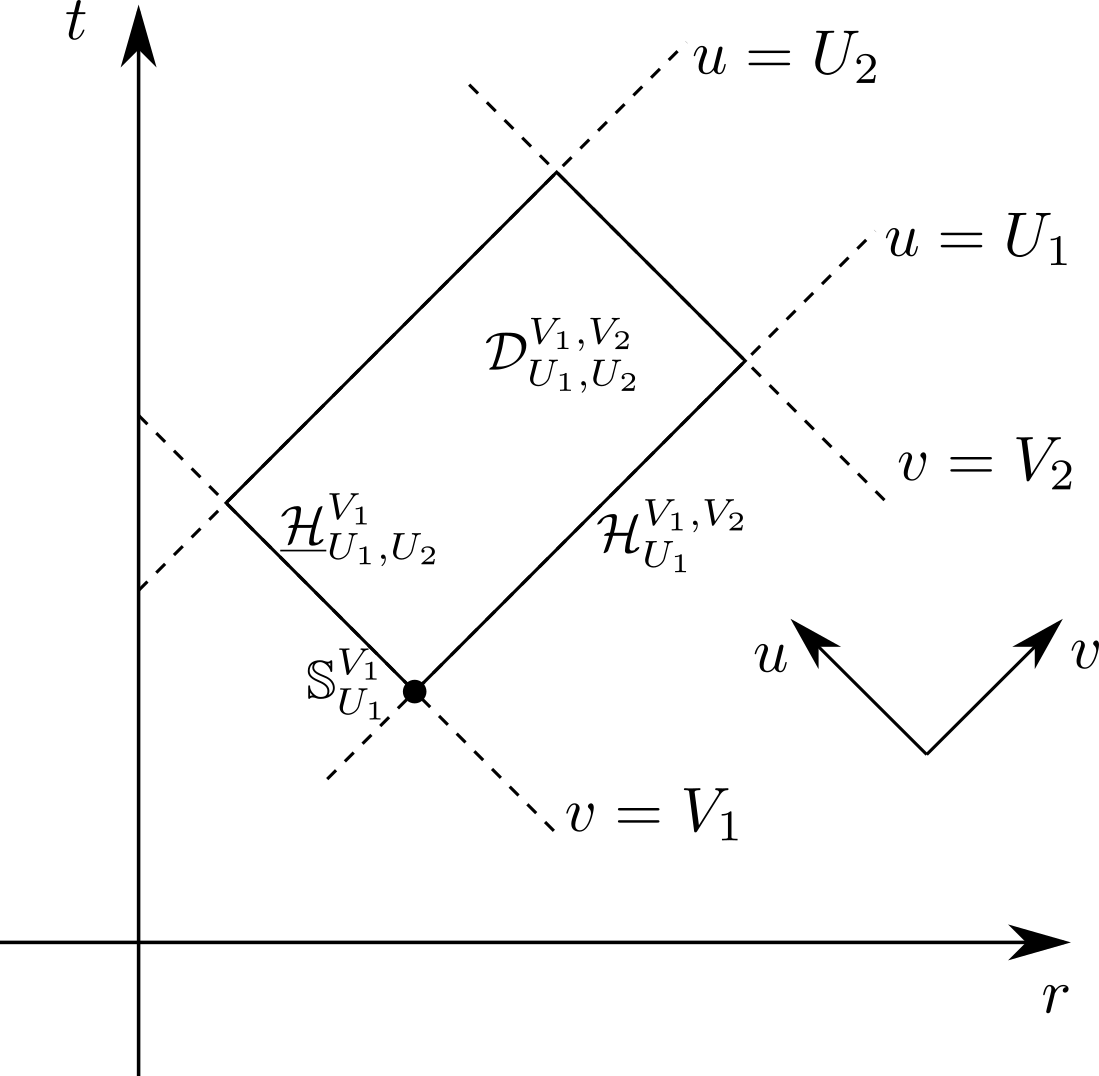}\hspace{70pt}
\includegraphics[width=0.3\textwidth]{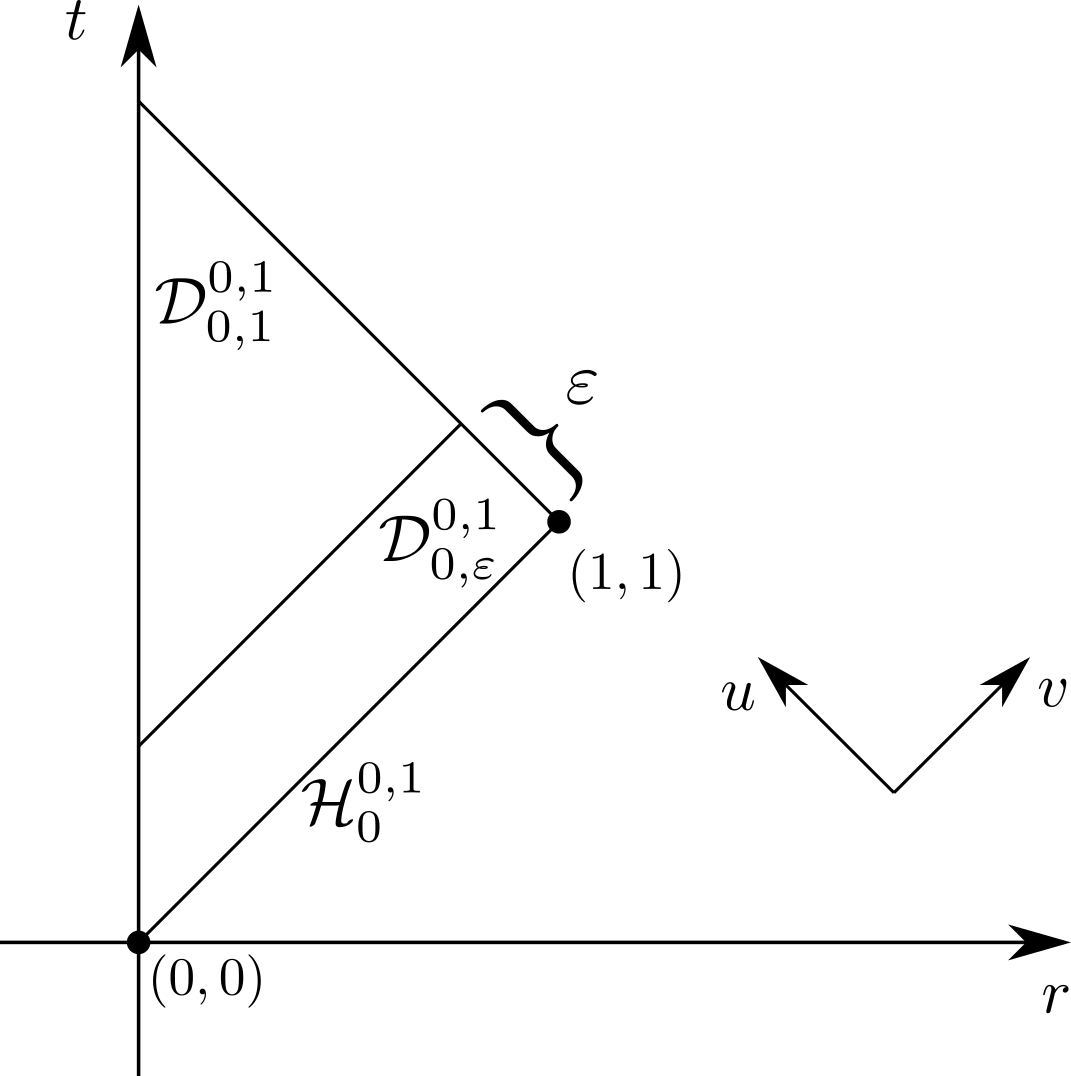}
\end{figure}
For simplicity, during the proof, we may omit the volume element in the integral.
For example 
\begin{Eq*}
\int_{\rS_U^V}f=\kl.\int_{|\omega|=1}f\cdot r^2\d\omega\kr|_{u=U,v=V}.
\end{Eq*}
We will also use the modified derivative. For any vector field $X$ in $\mathbb{R}^{1+3}$, $\hat{X}$ stands for $\hat{X}=X+r^{-1}X(r)$. More precisely, when acting on scalar field, it holds that 
\[
\hat{\pa}\phi=\pa \phi
+r^{-1}\pa(r)\phi,\quad \hat{L}\phi=L\phi+ r^{-1} \phi,\quad \hat{\Lb}\phi=\Lb\phi-r^{-1}\phi. 
\] 
For simplicity we may also use $\partial_{\omega} f$ to be short for $\Omega f$. For nonnegative integers $k, l$, the notation $f^{(k,l)}$ stands for taking sum over all $Z^{k_1}\Omega^{l_1}f$ with $k_1\leq k$ and $l_1\leq l$. For example
\begin{Eq*} 
\int_{\rHb_{0,U}^V}u|\hLb\phi_0^{(1,2)}|^2=\sum\limits_{k\leq 1, l\leq 2} \int_{\rHb_{0,U}^V}u|\hLb Z^k \Omega^l\phi_0 |^2.
\end{Eq*}

\subsection{Technical tools}
We will frequently use the following generalized  Hardy's type  inequality. 
\begin{lemma}
\label{lem:Hardy}
Assume that $f$ and $g$ are $C^1$ functions defined on the interval $[s_1,s_2]$ such that  $g$ is nonnegative and increasing. 
Then it holds that 
\begin{align}
|f(s_2)|^2g^{-1}(s_2)+g^{-2}(s_2)\int_{s_1}^{s_2}|f|^2g'\d s\leq 4 &|f(s_1)|^2g^{-1}(s_2)+2\int_{s_1}^{s_2}|f'|^2(g')^{-1}\d s,\label{Eq:kt1}\\%key tool 1
2|f(s_2)|^2g^{-1}(s_2)+\int_{s_1}^{s_2}|f|^2g^{-2}g'\d s\leq 2 &|f(s_1)|^2g^{-1}(s_1)+4\int_{s_1}^{s_2}|f'|^2(g')^{-1}\d s.\label{Eq:kt2}%key tool 2
\end{align}
The second inequality also holds for the case when $f(s_1)=g(s_1)=0$, that is, 
\begin{Eq}\label{Eq:kt3}%key tool 3 
2|f(s_2)|^2g^{-1}(s_2)+\int_{s_1}^{s_2}|f|^2g^{-2}g'\d s\leq 4 \int_{s_1}^{s_2}|f'|^2(g')^{-1}\d s.
\end{Eq}
\end{lemma}
\begin{proof}
See Lemma 6.1. and Corollary 6.2. in \cite{Yang:MKG:scattering}. However for readers interest, we sketch the proof here. For the first inequality,  without loss of generality, assume that 
\[
\int_{s_1}^{s_2}|f'|^2(g')^{-1}\d s=1.
\]
Then we can show that 
\begin{align*}
 &f^2|_{s=s_1}^{s_2}
  g^{-1}(s_2) + g^{-2}(s_2)\int_{s_1}^{s_2}|f|^2g'\d s \\ 
 &\leq g^{-1}(s_2) \left(\int_{s_1}^{s_2}|f'|\d s\right)^2+ g^{-2}(s_2) \int_{s_1}^{s_2}\left( |f(s_1)|+\int_{s_1}^s |f'(t)|\d t\right)^2g'(s)\d s\\ 
 &\leq g^{-1}(s_2) \int_{s_1}^{s_2} g'(s)\d s+2g^{-2}(s_2) \int_{s_1}^{s_2} ( |f(s_1)|^2+ \int_{s_1}^s g'(t)\d t  ) g'(s)\d s\\ 
 &\leq 1+ 2g^{-2}(s_2) |f(s_1)|^2 (g(s_2)-g(s_1))+ 2g^{-2}(s_2) \int_{s_1}^{s_2}     (g(s)-g(s_1)   ) g'(s)\d s \\ 
 &\leq 3+2 g^{-1}(s_2) |f(s_1)|^2.
\end{align*}
 This proves inequality \eqref{Eq:kt1}. 

 For the second inequality, integration by parts shows that 
 \begin{align*}
( |f|^2g^{-1})|_{s=s_1}^{s_2}
 +\int_{s_1}^{s_2}|f|^2g^{-2}g'\d s &= ( |f|^2g^{-1})|_{s_1}^{s_2} -\int_{s_1}^{s_2}|f|^2 \d g^{-1} \\ 
&\leq 2 \left(\int_{s_1}^{s_2} |f'|^2 (g')^{-1}\d s\right)^{\frac{1}{2}} \left(\int_{s_1}^{s_2} g^{-2}|f|^2 g' \d s\right)^{\frac{1}{2}}\\ 
&\leq 2  \int_{s_1}^{s_2} |f'|^2 (g')^{-1}\d s+\frac{1}{2}\int_{s_1}^{s_2} g^{-2}|f|^2 g' \d s,
 \end{align*}
from which we conclude that 
 \begin{align*}
2 |f(s_2)|^2g^{-1}(s_2) +\int_{s_1}^{s_2}|f|^2g^{-2}g'\d s
&\leq 2 |f(s_1)|^2g^{-1}(s_1)  +4  \int_{s_1}^{s_2} |f'|^2 (g')^{-1}\d s.
 \end{align*}
 We hence finished the proof for the Lemma. 

\end{proof}
The following inequality is straightforward. 
\begin{lemma}\label{La:kt4}%key tool 4
Let $f(s)$ be a positive function on the interval $[s_1,s_2]$ with $0\leq s_1$. 
If  there exist positive constants $C$ and $p_1$ such that
\begin{Eq*}
\int_{s_1}^{S}f(s)\d s\leq C S^{p_1},\quad \forall S\in[s_1,s_2],
\end{Eq*}
then, for all $p_2<p_1$ and $S\in[s_1,s_2]$, it holds that 
\begin{Eq*}
\int_{s_1}^{S}s^{-p_2}f(s)\d s\leq \frac{Cp_1}{p_1-p_2} S^{p_1-p_2}.
\end{Eq*}

\end{lemma}

\subsection{The vector field method}
The energy momentum tensor associated to a scalar field $f$ in $\mathbb{R}^{1+3}$ is given as follows
\begin{Eq*}
T[f]_{\mu\nu}:=\partial_\mu f\partial_\nu f-\frac{1}{2}\mm_{\mu\nu}\partial^{\sigma}f\partial_\sigma f.
\end{Eq*}
Under the null frame $\{L, \Lb, e_1, e_2\}$, we can compute that  
\begin{Eq*}
T[f]_{LL}=|Lf|^2,\quad T[f]_{L\Lb}=|\spartial f|^2, \quad T[f]_{\Lb\Lb}=|\Lb f|^2,\quad \sum\limits_{j=1}^2 T[f]_{jj}=|\spartial f|^2-\partial^\mu f\partial_\mu f.
\end{Eq*}
%which satisfies $\partial^\mu T[f]_{\mu\nu}=(\partial_\mu\partial^\mu f)(\partial_\nu f)$. 
For any vector field $X$ and any function $\chi$, define the current
\begin{Eq*}
J[f]_\mu:=T[f]_{\mu\nu}X^\nu-\frac{1}{2}f^2\partial_\mu\chi+\frac{1}{2}\chi\partial_\mu(f^2).
\end{Eq*}
We can compute that 
\begin{Eq*}
\partial^\mu J[f]_\mu=-(Xf+\chi f)\square f+T[f]_{\mu\nu}\pi^{\mu\nu}-\frac{1}{2}f^2\partial^\mu\partial_\mu\chi+\chi\partial^\mu f\partial_\mu f,
\end{Eq*}
where $\pi_{\mu\nu}^X=2^{-1}(\partial_\mu X_\nu+\partial_\nu X_\mu)$ is the deformation tensor of the vector field $X$. Then for any region $\rD$ in $\mathbb{R}^{1+3}$, Stokes' formula leads to the energy identity 
\begin{Eq*}
\int_{\rD}\partial^\mu J[f]_\mu=\int_{\partial\rD}J[f]_{\vec n},
\end{Eq*}
where $\vec n$ is the unit outer normal vector of the boundary $\partial\rD$.

In application, let 
$$X=2vL+u\Lb,\quad \chi=r^{-1}X^r=r^{-1}(v+r).$$
The nonvanishing components of the deformation tensor $\pi^X$ are
\begin{Eq*}
\pi_{L\Lb}^X
=-3,\qquad \pi_{jj}^X= \frac{v+r}{r},\quad j=1, 2.
\end{Eq*}
We thus can compute that  
\begin{Eq*}
J[f]_L=&2v|Lf|^2+u|\spartial f|^2-\frac{1}{2}f^2L\kl(\frac{v+r}{r}\kr)+\frac{v+r}{2r} L(f^2)\\
=&2v|\hL f|^2+u|\spartial f|^2-r^{-2}L\kl(\frac{r(3t+r)}{4} f^2\kr),\\
J[f]_\Lb=&2v|\spartial f|^2+u|\Lb f|^2-\frac{1}{2}f^2\Lb\kl(\frac{v+r}{r}\kr)+\frac{v+r}{2r} \Lb(f^2)\\
=&2v|\spartial f|^2+u|\hLb\phi|^2+r^{-2}\Lb\kl(\frac{r(3t+r)}{4}f^2\kr),\\
\partial^\mu J[f]_\mu=&-\big(2v(L+r^{-1})f+u(\Lb-r^{-1})f\big)\square f-\frac{3}{4}|\spartial f|^2+\frac{v+r}{r}\kl(|\spartial f|^2-\partial^\mu f\partial_\mu f\kr)\\
&\quad -\frac{1}{2}f^2\Delta\kl(\frac{t}{2r}\kr)+\frac{v+r}{r}\partial^\mu f\partial_\mu f\\\
=&-(2v\hL f+u\hLb f)\square f+\frac{t}{2r}|\spartial f|^2+\pi tf^2\delta(x).
\end{Eq*}
Now take the  region to be $\rD_{0,U}^{0,V}$ with boundary consisting of null hypersurfaces. Note that those terms 
\[
-r^{-2}L\kl(\frac{r(3t+r)}{4} f^2\kr),\quad r^{-2}\Lb\kl(\frac{r(3t+r)}{4}f^2\kr)
\]
are canceled after integration by parts on the boundary. We therefore can derive the following energy inequality
\begin{Eq}\label{Eq:Mei}%Main energy inequality
&\int_{\rH_U^{U,V}}2v|\hL f|^2+u|\spartial f|^2+\int_{\rHb_{0,U}^V}2v|\spartial f|^2+u|\hLb f|^2+\int_{\rD_{0,U}^{0,V}}\frac{v}{2r}|\spartial f|^2+\int_0^{2U} [tf^2]_{x=0}\d t\\
%&\quad+[rtf^2]_{t=2U,x=0}\\
\leq & \int_{\rD_{0,U}^{0,V}} | \hat{X}f |\cdot|\square f|+\int_{\rH_{0}^{0,V}}2 v|\hL f|^2+u|\spartial f|^2.
\end{Eq}
For fixed $U$ and $V$, define the associated energy of $f$ 
\begin{Eq*}
E[f](U,V):=\sup_{u\leq U}\int_{\rH_u^{u,V}}2v|\hL f|^2+u|\spartial f|^2+\sup_{v\leq V}\int_{\rHb_{0,U}^v}2v|\spartial f|^2+u|\hLb f|^2+\int_{\rD_{0,U}^{0,V}}\frac{v}{2r}|\spartial f|^2.
\end{Eq*}

\section{Proof for the main \Tm{Tm:M1}}
\label{sec:Th1}
The local solution is constructed via the standard iteration process, for which the first step is to  understand the linear solution. To show that the solution exists for the characteristic initial value problem to the linear wave equation (equation \eqref{Eq:Me} with given $F$), first notice that we have precise representation formula for the solution 
\begin{Eq*}
\phi_0(t,x)=r^{-1}\kl(\frac{t+r}{2}\varphi_0(\frac{t+r}{2})-\frac{t-r}{2}\varphi_0(\frac{t-r}{2})\kr)
\end{Eq*}
 for the case when $F=0$ and the initial data $\varphi_0$ is spherically symmetric. Next still for the linear wave equation with $F=0$, by   using the spherical average method, we are able to obtain the solution $\phi_0$ at $r=0$.  
Then the solution can be represented as follows
\begin{Eq*}
\phi_0(t, x)=\frac{1}{4\pi}\int_{S^2} \kl[\partial_v(v\varphi_0)\kr](y)\d\sigma_\eta, 
\end{Eq*}
in which 
\begin{Eq*}
  y=\frac{r^{-2}(t-\tau)(x\cdot\eta)x+x+\tau\eta}{2},\quad \tau=\sqrt{t^2-r^2}
\end{Eq*}
by using Lorentz transformation. See \cite{Friedlander75:wave:curved} for detailed discussion. 

Finally for the general inhomogeneous linear wave equation with given $F$, we can split the solution into the linear part $\phi_0$ and the remaining part $\phi-\phi_0$ with vanishing initial data on the cone $\mathcal{H}_0$. For this part, we can extend $F$ to the whole space by simply setting $F=0$ outside the cone. Then the solution $\phi-\phi_0$ can be represented through the Kirchhoff formula for the Cauchy problem for linear wave equation with vanishing initial data and such extended force term $F$. 
\subsection{Energy estimate for the linear solution}
Since we are using the vector field energy method, for such singular initial data, we need first obtained necessary bounds for the linear solution. 
\begin{lemma}
For $0\leq U,V \leq 1$, we have the weighted energy estimate for the linear solution 
\begin{Eq}\label{Eq:Eoeols}%Estimate of energy of linear solution
E[\phi_0^{(1,2)}](U,V)\lesssim& M_0^2V^{2\delta}.
\end{Eq}
The energy flux through the incoming null hypersurface obtains the bound 
\begin{Eq}\label{Eq:EoLeols}%Estimate of Lb energy of linear solution
\int_{\rHb_{0,U}^V}u|\hLb\phi_0^{(1,2)}|^2\lesssim M_0^2 U^\delta V^{\delta}.
\end{Eq}
Moreover the linear solution $\phi_0$ verifies the following estimates
\begin{align}
\int_U^V\int_{\rS^2}|\phi_0^{(1,2)}|_{u=U}^2\d\omega\d v\lesssim& M_0^2U^{-1}V^{2\delta},\label{Eq:Eolsi}\\%Estimate of linear solution itself
\|\phi_0^{(1,2)}\|_{L_\omega^2}\lesssim& M_0 r^{-1/2}v^{\delta-1/2},\label{Eq:Eolsos1}\\%Estimate of linear solution on sphere 1
\|\phi_0^{(0,2)}\|_{L_\omega^2}\lesssim& M_0 v^{\delta-1}.\label{Eq:Eolsos2}%Estimate of linear solution on sphere 2
\end{align}
on the outgoing null hypersurface $\mathcal{H}_{U}^{U, V}$ and the two sphere $\mathbb{S}_{U}^{V}$ respectively. 
\end{lemma}

\begin{proof}
The proof for the energy estimate  \eqref{Eq:Eoeols} is standard.
Commuting the linear equation with the vector fields  $Z$ and $\Omega$, we obtain the equation 
\begin{Eq*}
 \phi_0^{(1,3)}=0, \quad 
\phi_0^{(1,3)}|_{u=0}=\varphi_0^{(1,3)}.
\end{Eq*}
The assumption \eqref{Eq:Cod} on the initial data $\varphi_0$ in particular implies that 
$$\lim_{v\rightarrow 0}v\varphi_0^{(1,3)}(v,\omega)=0. $$
Hence in view of the energy estimate \eqref{Eq:Mei} applied to the scalar field $\phi_0^{1, 3}$, we obtain that  
\begin{Eq*}
E[\phi_0^{(1,3)}](U,V)\lesssim E[\phi_0^{(1,3)}](0,V)\lesssim M_0^2V^{2\delta}.
\end{Eq*}
This gives the energy estimate  \eqref{Eq:Eoeols}. In particular the energy estimate \eqref{Eq:EoLeols} through the incoming null hypersurface holds for the case when $V\leq 2U$. Moreover the solution verifies the bound 
$$\lim\limits_{r\rightarrow 0}  r\phi_0^{(1,2)}(t,x)  =0.$$
For the case when $V\geq 2U$, we make use of the Hardy's type of inequality  \eqref{Eq:kt1}  in Lemma \ref{lem:Hardy}. First by using the equation, we observe that 
\begin{Eq*}
L\Lb(r Z^k \Omega^l\phi_0)=r^{-1}\Delta_{\rS^2}Z^k \Omega^l\phi_0 =   \spartial Z^k \Omega^{l+1} \phi_0.
\end{Eq*} 
Under the assumption that $2U\leq V$, on the incoming null hypersurface $\rHb_{0,U}^V $, it holds that 
\[
t\leq 3r,\quad \frac{r}{2}\leq v\leq 2r.
\]
Comparing with the energy flux through the incoming null hypersurface $\rHb_{0, U}^{2U}$, we therefore can estimate that 
\begin{Eq}\label{Eq:Eolswls}%Estimate of linear solution with low singularity
\int_{\rHb_{0,U}^V}u|\hLb\phi_0^{(1,2)}|^2= &\int_0^U\int_{\rS^2}u|\Lb(r\phi_0^{(1,2)})|_{v=V}^2\d\omega\d u\\
\lesssim&\int_0^U\int_{\rS^2}u|\Lb(r\phi_0^{(1,2)})|_{v=2U}^2\d\omega\d u\\
&\quad +V^\delta\int_{2U}^V\int_0^U\int_{\rS^2}uv^{1-\delta}|L\Lb(r\phi_0^{(1,2)})|^2\d\omega\d u\d v\\
\lesssim&E[\phi_0^{(1,2)}](U,2U)+UV^\delta\int_{2U}^Vv^{-\delta-2}\int_{\rHb_{0,U}^v}v|\spartial\phi_0^{(1,3)}|^2\d v \\
\lesssim& M_0^2 U^{2\delta}+UV^\delta\int_{2U}^Vv^{-\delta-2}(M_0^2v^{2\delta})\d v\\
\lesssim& M_0^2 U^\delta V^{\delta}.
\end{Eq}
This completes the proof for the energy estimate  \eqref{Eq:EoLeols}.

For estimate \eqref{Eq:Eolsi}, first notice that 
$$\kl.r\phi_0^{(1,2)}(t,x)\kr|_{v=u}=\kl.r\phi_0^{(1,2)}(t,x)\kr|_{r=0}=0.$$ 
In view of the Hardy's type estimate \eqref{Eq:kt3}, we can derive that 
\begin{Eq*}
&(V-U)^{-1}\int_{\rS^2}|r\phi_0^{(1,2)}|\low{u=U\\v=V}^2\d\omega+\int_U^V\int_{\rS^2} (v-U)^{-2}|r\phi_0^{(1,2)}|\low{u=U}^2\d\omega\d v\\
\lesssim&\int_U^V\int_{\rS^2} |L(r\phi_0^{(1,2)})|\low{u=U}^2\d\omega\d v\\
\lesssim& U^{-1}\int_{\rH_U^{U,V}} v|\hL \phi_0^{(1,2)}|^2\\
\lesssim& M_0^2 U^{-1}V^{2\delta}.
\end{Eq*}
Here recall that on the out going null hypersurface $\mathcal{H}_{U}^{U, V}$ with fixed $u=U$, it holds that $v-U=v-u=r$. On the other hand, on the  two-sphere $\mathbb{S}_U^V$, using \eqref{Eq:kt3} again and in view of Lemma  \La{La:kt4} together  with the energy estimate \eqref{Eq:Eoeols}, we can show that 
\begin{Eq*}
V^{-\delta}\int_{\rS_U^V}|r\phi_0^{(1,2)}|^2\d\omega%+\int_U^V\int_{\rS^2} v^{-1-\delta}|r\phi_0^{(1,3)}|\low{u=U}^2\d\omega\d v\\
\lesssim&\int_U^V\int_{\rS_U^v} v^{1-\delta}|L(r\phi_0^{(1,2)})|^2\d\omega\d v\\
\lesssim& M_0^2 V^\delta.
\end{Eq*} 
Combining this with the previous estimate, we in particular can conclude that 
\begin{align*}
\int_{\rS_U^V}|r\phi_0^{(1,2)}|^2\d\omega \lesssim M_0^2 (V-U)^{-1}V^{2\delta-1} \min\{(V-U)^{-1} V, U^{-1}V \}\lesssim M_0^2 (V-U)^{-1}V^{2\delta-1}.
\end{align*}
This proves the estimate \eqref{Eq:Eolsos1}. Consequently the estimate \eqref{Eq:Eolsos2} holds for the case when $V\geq 2U$. For the case when  $V\leq 2U$ regarding the last estimate \eqref{Eq:Eolsos2}, note that 
\begin{Eq*} 
 v\partial_v=2^{-1}(S+\omega^i L_i).  %,u\partial_u=2^{-1}(S-\omega^iB_i), \qquad
\end{Eq*}
By using the Hardy's type inequality \eqref{Eq:kt1}, we can estimate that 
\begin{Eq*}%\label{Eq:Eolswns}%Estimate of linear solution with no singularity
V\int_{\rS^2}|\phi_0^{(0,2)}|\low{u=U\\v=V}^2\d\omega
\lesssim& V\int_{\rS^2}|\phi_0^{(0,2)}|\low{u=U\\v=2U}^2\d\omega+\int_{V}^{2U}\int_{\rS^2}|vL\phi_0^{(0,2)}|\low{u=U}^2\d\omega\d v\\
\lesssim& U\|\phi_0^{(1,2)}(U,2U)\|_{L_\omega^2}^2+\int_{U}^{2U}\int_{\rS^2}|\phi_0^{(1,2)}|\low{u=U}^2\d\omega\d v\\
\lesssim& M_0^2U^{2\delta-1}.
\end{Eq*}
We hence finished the proof for the Lemma. 
\end{proof}

\subsection{Energy estimate for the iteration}
To construct the local solution for the characteristic initial value problem, we use the standard iteration process. 
For $k\geq 1$, let  $\phi_k$ be the solution to the linear equation 
\begin{Eq*}
\square \phi_k=F_{k-1}:=F(\phi_{k-1},\partial\phi_{k-1}),\qquad
\phi_{k}|_{u=0}=\varphi_0.
\end{Eq*}
We also denote the difference $w_k:=\phi_k-\phi_0$. In particular  $w_k$ solves the equation 
\begin{Eq*}
\square w_k=F_{k-1},\qquad
w_{k}|_{u=0}=0.
\end{Eq*}
We derive the key estimates  for the linear solution $w_k$. 
\begin{proposition}
\label{La:Eoi}%Estimate of iteration
There exists a positive constant $M_1$ and a small positive constant  $\varepsilon < 1$ such that for all $k\geq 0$ and  $0\leq U\leq \min\{\varepsilon, V\}\leq 1$, the linear solution $w_k$ verifies the following weighted energy estimate 
\begin{Eq}\label{Eq:Eoeons}%Estimate of energy of nonlinear solution
E[w_k^{(1,2)}](U,V)\leq M_0^2U^\delta V^{\delta}e^{2M_1V^{2\delta}}.
\end{Eq}
Moreover we have the improved weighted energy estimate through the out going null hypersurface 
\begin{align}
\int_U^V\int_{\rS^2}|w_k^{(1,2)}|_{u=U}^2\d\omega\d v\lesssim& M_0^2e^{2M_1}U^{\delta-1}V^{\delta}\label{Eq:Eonsi}.%Estimate of nonlinear solution itself
\end{align}
Consequently the solution is bounded on the two sphere in the following sense
\begin{align}
\|w_k^{(1,2)}\|_{L_\omega^2}\lesssim& M_0 e^{M_1} r^{-1/2}u^{\delta/2}v^{(\delta-1)/2},\label{Eq:Eonsos1}\\%Estimate of nonlinear solution on sphere 1
\|w_k^{(0,2)}\|_{L_\omega^2}\lesssim& M_0 e^{M_1} u^{\delta/2}v^{\delta/2-1}.\label{Eq:Eonsos2}%Estimate of nonlinear solution on sphere 2
\end{align}
\end{proposition}
\begin{proof}
We first show that estimates \eqref{Eq:Eonsi}, \eqref{Eq:Eonsos1} and \eqref{Eq:Eonsos2} follow from the weighted energy estimate  \eqref{Eq:Eoeons}.
Assuming that \eqref{Eq:Eoeons} holds, similar to the proof of the estimate \eqref{Eq:Eolsi}, we can show that  
\begin{Eq*}
&(V-U)^{-1}\int_{\rS^2}|rw_k^{(1,2)}|\low{u=U\\v=V}^2\d\omega+\int_U^V\int_{\rS^2} |w_k^{(1,2)}|\low{u=U}^2\d\omega\d v\\
&\lesssim U^{-1}\int_{\rH_U^{U,V}} v|\hL w_k^{(1,2)}|^2\\
&\lesssim M_0^2e^{2M_1} U^{\delta-1}V^\delta .
\end{Eq*}
This implies estimate \eqref{Eq:Eonsi}. 
Consequently by using the Hardy's type inequality \eqref{Eq:kt3} and \La{La:kt4}, we obtain that 
\begin{Eq*}
U^{-\frac{\delta}{2}}\int_{\rS^2}|rw_k^{(1,2)}|\low{u=U\\v=V}^2\d\omega\lesssim&\int_0^U\int_{\rS^2} u^{1-\frac{\delta}{2}}|\Lb(rw_k^{(1,2)})|\low{v=V}^2\d\omega\d u\\
\lesssim& M_0^2 e^{2M_1} U^{\frac{\delta}{2}} V^{\delta}.
\end{Eq*} 
Combining this estimate with the previous inequality leads to the estimate  \eqref{Eq:Eonsos1} and the estimate \eqref{Eq:Eonsos2} for the case when $V\geq 2U$. 
For the case when $V\leq 2U$, similar to the proof for \eqref{Eq:Eolsos2}, we can show that 
\begin{Eq*}
V\int_{\rS^2}|w_k^{(0,2)}|\low{u=U\\v=V}^2\d\omega
\lesssim& V\int_{\rS^2}|w_k^{(0,2)}|\low{u=U\\v=2U}^2\d\omega+\int_{V}^{2U}\int_{\rS^2}|vLw_k^{(0,2)}|\low{u=U}^2\d\omega\d v\\
\lesssim& U\|w_k^{(1,2)}(U,2U)\|_{L_\omega^2}^2+\int_{U}^{2U}\int_{\rS^2}|w_k^{(1,2)}|\low{u=U}^2\d\omega\d v\\
\lesssim& M_0^2e^{2M_1}U^{2\delta-1}.
\end{Eq*}
We thus finished the proof for \eqref{Eq:Eonsos2}.

\bigskip 

The rest of the argument focuses on the proof of the weighted energy estimate \eqref{Eq:Eoeons} for all  $k\geq 0$. We prove by induction. Obviously estimate \eqref{Eq:Eoeons} holds for the case when  $k=0$ since $w_0=0$. Now assume that  it holds for the case when  $k=K$. We show that the estimate holds for $k=K+1$. 
First from the energy estimate  \eqref{Eq:Mei}, we can obtain that  
\begin{Eq*}
E[w_{K+1}^{(1,2)}](U,V)\leq& C\int_{\rD_{0,U}^{0,V}}\kl|(v\hL+u\hLb)w_{K+1}^{(1,2)}\kr|\cdot|F_K^{(1,2)}|\\
\leq& CV^{\frac{\delta}{2}}\int_{\rD_{0,U}^{0,V}}uv^{1-\frac{\delta}{2}}|F_K^{(1,2)}|^2+\frac{\delta}{8}V^{-\frac{\delta}{2}} \int_0^Vv^{\frac{\delta}{2}-1}\int_{\rHb_{0,U}^{v}}u|\hLb w_{K+1}^{(1,2)}|^2\d v\\
&\quad+C U^{\frac{\delta}{2}}\int_{\rD_{0,U}^{0,V}}u^{1-\frac{\delta}{2}}v|F_K^{(1,2)}|^2+\frac{\delta}{8}U^{-\frac{\delta}{2}} \int_0^Uu^{\frac{\delta}{2}-1}\int_{\rH_u^{u,V}}v|\hL w_{K+1}^{(1,2)}|^2\d u\\
\leq&CU^{\frac{\delta}{2}}V^{\frac{\delta}{2}}\int_{\rD_{0,U}^{0,V}}u^{1-\frac{\delta}{2}}v^{1-\frac{\delta}{2}}|F_K^{(1,2)}|^2+ \frac{1}{2}E[w_{K+1}^{(1,2)}](U,V).
\end{Eq*}
Here by our notation the constant $C$ relies only on $M_0$ and $\delta$. We see that the second term on right hand side can be absorbed. It thus remains to understand the nonlinear term $F_K^{(1,2)}$.

For the quadratic term in $F$, notice that  
\begin{Eq*}
|(f\partial f)^{(1,2)}|\lesssim& \sum\Low{n_{11}+n_{12}=1\\n_{21}+n_{22}=2}|f^{(n_{11},n_{21})}||(\partial f)^{(n_{12},n_{22})}|.
%\lesssim&\sum\Low{n_{11}+n_{12}=1\\n_{21}+n_{22}=2}|f^{(n_{11},n_{21})}||\partial f^{(n_{12},n_{22})}|,\\
%|(f^3)^{(1,2)}|\lesssim& \sum\Low{n_{11}+n_{12}+n_{13}=1\\n_{21}+n_{22}+n_{23}=2}|f^{(n_{11},n_{21})}||f^{(n_{12},n_{22})}||f^{(n_{13},n_{23})}|.
\end{Eq*}
Using H\"older's inequality together with  the Sobolev inequality on the sphere, we can show that 
\begin{Eq*}
\int_{\rS^2}|(f\partial f)^{(1,2)}|^2\d\omega\lesssim&\sum_{n_{11}+n_{12}=1}\|f^{(n_{11},2)}\|_{L_\omega^2}^2 \|(\partial f)^{(n_{12},2)}\|_{L_\omega^2}^2\\
\lesssim&\|f^{(0,2)}\|_{L_\omega^2}^2 \|\partial f^{(1,2)}\|_{L_\omega^2}^2+\|f^{(1,2)}\|_{L_\omega^2}^2 \|\partial f^{(0,2)}\|_{L_\omega^2}^2.
\end{Eq*}
Similarly for the cubic term, we have 
\begin{Eq*}
\int_{\rS^2}|(f^3)^{(1,2)}|^2\d\omega\lesssim&\|f^{(0,2)}\|_{L_\omega^2}^4 \|f^{(1,2)}\|_{L_\omega^2}^2.
\end{Eq*}
We are now ready to bound the nonlinear term
\begin{Eq*}
&\int_{\rD_{0,U}^{0,V}}u^{1-\frac{\delta}{2}}v^{1-\frac{\delta}{2}}|F_K^{(1,2)}|^2\\
\lesssim&\int_0^U\int_u^Vu^{1-\frac{\delta}{2}}v^{1-\frac{\delta}{2}}\kl(\int_{\rS^2}|(\phi_K\partial \phi_K)^{(1,2)}|^2+|(\phi_K^3)^{(1,2)}|^2\d\omega\kr) r^2\d v\d u\\
%\lesssim&\sum\Low{n_{21}+n_{22}=2}\int_{\rD_{0,U}^{0,V}}u^{1-\frac{\delta}{2}}v^{1-\frac{\delta}{2}}|\phi_K^{(0,n_{21})}|^2|\partial \phi_K^{(1,n_{22})}|^2\\
%&\quad+\sum_{n_{21}+n_{22}=2}\int_{\rD_{0,U}^{0,V}}u^{1-\frac{\delta}{2}}v^{1-\frac{\delta}{2}}|\phi_K^{(1,n_{21})}|^2|\partial \phi_K^{(0,n_{22})}|^2\\
%&\quad+\sum\Low{n_{11}+n_{12}+n_{13}=1\\n_{21}+n_{22}+n_{23}=2}\int_{\rD_{0,U}^{0,V}}u^{1-\frac{\delta}{2}}v^{1-\frac{\delta}{2}}|\phi_K^{(n_{11},n_{21})}|^2|\phi_K^{(n_{12},n_{22})}|^2|\phi_K^{(n_{13},n_{23})}|^2\\
\lesssim&\int_0^U\int_u^V u^{1-\frac{\delta}{2}}v^{1-\frac{\delta}{2}}\|\phi_K^{(0,2)}\|_{L_\omega^2}^2\|(\hL,\hLb,\spartial,r^{-1}) \phi_K^{(1,2)}\|_{L_\omega^2}^2 r^2\d v\d u\\
&\quad+\int_0^U\int_u^V u^{1-\frac{\delta}{2}}v^{1-\frac{\delta}{2}}\|\phi_K^{(1,2)}\|_{L_\omega^2}^2\| (L,u^{\frac{1}{2}}v^{-\frac{1}{2}}\Lb,r^{\frac{1}{2}}v^{-\frac{1}{2}}\hLb,\spartial,r^{-1})\phi_K^{(0,2)}\|_{L_\omega^2}^2 r^2\d v\d u\\
&\quad+\int_0^U\int_u^Vu^{1-\frac{\delta}{2}}v^{1-\frac{\delta}{2}}\|\phi_K^{(0,2)}\|_{L_\omega^2}^4\|\phi_K^{(1,2)}\|_{L_\omega^2}^2r^2\d v\d u\\
\equiv&I_1+I_2+I_3.
\end{Eq*}
Here $I_j$ represents the integral on the  $j$'s line on the right hand side of the second last inequality.  

For the $L$ derivative part of $I_1$, in view of the estimates \eqref{Eq:Eolsos2} for $\phi_0$ % 修改了
and \eqref{Eq:Eonsos2} for $w_K$, we can show that 
\begin{Eq*}
I_{1,1}:=&\int_0^U\int_u^V u^{1-\frac{\delta}{2}}v^{1-\frac{\delta}{2}}\|\phi_K^{(0,2)}\|_{L_\omega^2}^2\|\hL\phi_K^{(1,2)}\|_{L_\omega^2}^2 r^2\d v\d u\\
\lesssim&\int_0^U\int_{\rH_u^{u,V}} u^{1-\frac{\delta}{2}}v^{1-\frac{\delta}{2}}(M_0^2e^{2M_1}v^{2\delta-2})|\hL\phi_K^{(1,2)}|^2 \d u\\
\lesssim&M_0^2e^{2M_1}\int_0^U u^{\delta-1}\int_{\rH_u^{u,V}} v|\hL\phi_K^{(1,2)}|^2 \d u.
\end{Eq*}
Now by using the estimate \eqref{Eq:Eoeols} for $\phi_0$ and
 \eqref{Eq:Eoeons} for $w_K$, we obtain that 
\begin{Eq*}
I_{1,1}\lesssim&M_0^2e^{2M_1}\int_0^U u^{\delta-1}(e^{2M_1V^{2\delta}}M_0^2V^{2\delta}) \d u\\
\lesssim&M_0^4e^{4M_1}U^{\delta}V^{2\delta}.
\end{Eq*}
For the $\hLb$ derivative part  of $I_1$, again by using the estimates  \eqref{Eq:Eolsos2} and \eqref{Eq:Eonsos2}, 
we have
\begin{Eq*}
I_{1,2}:=&\int_0^U\int_u^V u^{1-\frac{\delta}{2}}v^{1-\frac{\delta}{2}}\|\phi_K^{(0,2)}\|_{L_\omega^2}^2\|\hLb\phi_K^{(1,2)}\|_{L_\omega^2}^2 r^2\d v\d u\\
\lesssim&\int_0^V\int_{\rHb_{0,U}^{v}} u^{1-\frac{\delta}{2}}v^{1-\frac{\delta}{2}}(M_0^2v^{2\delta-2}+M_0^2e^{2M_1}u^\delta v^{\delta-2})|\hLb\phi_K^{(1,2)}|^2 \d v\\
\lesssim&V^{\frac{\delta}{2}}\int_0^V (M_0^2v^{\delta-1}+M_0^2e^{2M_1}U^\delta v^{-1}) \int_{\rHb_{0,U}^{v}} u^{1-\frac{\delta}{2}}|\hLb\phi_K^{(1,2)}|^2 \d v.\\
\end{Eq*}
Then by using  \La{La:kt4} together with the estimates  \eqref{Eq:Eoeols} and \eqref{Eq:Eoeons}, 
we can show that 
\begin{Eq*}
I_{1,2}\lesssim&V^{\frac{\delta}{2}}\int_0^V (M_0^2v^{\delta-1}+M_0^2e^{2M_1}U^\delta v^{-1}) \kl(M_0^2U^{\frac{\delta}{2}}v^\delta e^{2M_1v^{2\delta}}\kr) \d v\\
\lesssim&V^{\frac{\delta}{2}}\int_0^V M_0^4e^{2M_1v^{2\delta}}U^{\frac{\delta}{2}}v^{2\delta-1}+M_0^4e^{4M_1}U^\frac{3\delta}{2} v^{\delta-1}\d v\\
\lesssim&\frac{M_0^4}{M_1} e^{2M_1V^{2\delta}}U^{\frac{\delta}{2}}V^{\frac{\delta}{2}}+M_0^4e^{4M_1}U^{\frac{3\delta}{2}}V^{\frac{3\delta}{2}}.
\end{Eq*}
For the angular derivative $\spartial$ part of $I_1$, similarly we can bound that 
\begin{Eq*}
I_{1,3}:=&\int_0^U\int_u^V u^{1-\frac{\delta}{2}}v^{1-\frac{\delta}{2}}\|\phi_K^{(0,2)}\|_{L_\omega^2}^2\|\spartial\phi_K^{(1,2)}\|_{L_\omega^2}^2 r^2\d v\d u\\
%\lesssim&\int_0^U\int_{\rH_u^{u,V}} u^{1-\frac{\delta}{2}}v^{1-\frac{\delta}{2}}(M_0^2e^{M_1}v^{2\delta-2})|\spartial\phi_K^{(1,2)}|^2 \d u\\
\lesssim&M_0^2e^{2M_1}\int_0^U u^{\delta-1}\int_{\rH_u^{u,V}} u|\spartial\phi_K^{(1,2)}|^2 \d u\\
%\lesssim&M_0^2e^{M_1}\int_0^U u^{\delta-1}(M_0^2}e^{M_1}V^{2\delta) \d u\\
\lesssim&M_0^4e^{4M_1}U^{\delta}V^{2\delta}.
\end{Eq*}
Finally, for the lower order term  $r^{-1}\phi_K$ in $I_1$, we  rely on the estimates \eqref{Eq:Eolsi} for $\phi_0$ 
and \eqref{Eq:Eonsi} for $w_K$ to conclude that 
\begin{Eq*}
I_{1,4}:=&\int_0^U\int_u^V u^{1-\frac{\delta}{2}}v^{1-\frac{\delta}{2}}\|\phi_K^{(0,2)}\|_{L_\omega^2}^2\|r^{-1}\phi_K^{(1,2)}\|_{L_\omega^2}^2 r^2\d v\d u\\
\lesssim&M_0^2e^{2M_1} \int_0^U u^{\delta}\int_u^V \int_{\rS^2}|\phi_K^{(1,2)}|^2 \d\omega\d v\d u\\
\lesssim&M_0^2e^{2M_1} \int_0^U u^{\delta}\kl(M_0^2e^{2M_1}u^{-1}V^{2\delta}\kr)\d u\\
\lesssim&M_0^4e^{4M_1}U^\delta V^{2\delta}.
\end{Eq*}
Next we estimate the term $I_2$. First notice that 
\begin{align*}
v|Lf^{(0,2)}|+u |\Lb f^{(0,2)}|\lesssim  |Sf^{(0,2)}|+\sum\limits_{j=1}^3|L_j f^{(0, 2)}|+|\Omega f^{(0, 2)}|.  
\end{align*}
In view of the estimates  \eqref{Eq:Eolsos1} for $\phi_0$ 
and \eqref{Eq:Eonsos1} for $w_K$, we can show that 
\begin{Eq*}
I_{2,1}:=&\int_0^U\int_u^V u^{1-\frac{\delta}{2}}v^{1-\frac{\delta}{2}}\|\phi_K^{(1,2)}\|_{L_\omega^2}^2\| (L,u^{\frac{1}{2}}v^{-\frac{1}{2}}\Lb)\phi_K^{(0,2)}\|_{L_\omega^2}^2 r^2\d v\d u\\
\lesssim&\int_0^U\int_u^V u^{1-\frac{\delta}{2}}v^{1-\frac{\delta}{2}}(M_0^2e^{2M_1}r^{-1} v^{2\delta-1})(M_0^2e^{2M_1} r^{-1}u^{-1}v^{2\delta-2}) r^2\d v\d u\\
\lesssim&M_0^4 e^{4M_1}U^{3\delta}.
\end{Eq*}
For the $r^{\frac{1}{2}}v^{-\frac{1}{2}}\hLb$ derivative part if $I_2$, 
similarly we can bound that  
\begin{Eq*}
I_{2,2}:=&\int_0^U\int_u^V u^{1-\frac{\delta}{2}}v^{1-\frac{\delta}{2}}\|\phi_K^{(1,2)}\|_{L_\omega^2}^2\| r^{\frac{1}{2}}v^{-\frac{1}{2}}\hLb\phi_K^{(0,2)}\|_{L_\omega^2}^2 r^2\d v\d u\\
\lesssim&\int_0^V\int_{\rHb_{0,U}^v}r u^{1-\frac{\delta}{2}}v^{-\frac{\delta}{2}}(M_0^2r^{-1}v^{2\delta-1}+M_0^2e^{M_1}r^{-1}u^\delta v^{\delta-1})| \hLb\phi_K^{(0,2)}|^2 \d v\\
\lesssim&\frac{M_0^4}{M_1} e^{2M_1V^{2\delta}}U^{\frac{\delta}{2}}V^{\frac{\delta}{2}}+M_0^4e^{4M_1}U^{\frac{3\delta}{2}}V^{\frac{3\delta}{2}}.
\end{Eq*}
For the angular derivative  $\spartial$ part in  $I_2$, 
by using the estimates \eqref{Eq:Eoeols} and \eqref{Eq:Eoeons}, 
we can show that 
\begin{Eq*}
I_{2,3}:=&\int_0^U\int_u^V u^{1-\frac{\delta}{2}}v^{1-\frac{\delta}{2}}\|\phi_K^{(1,2)}\|_{L_\omega^2}^2\| \spartial\phi_K^{(0,2)}\|_{L_\omega^2}^2 r^2\d v\d u\\
\lesssim&\int_{\rD_{0,U}^{0,V}} u^{1-\frac{\delta}{2}}v^{1-\frac{\delta}{2}}(M_0^2e^{2M_1}r^{-1} v^{2\delta-1})|\spartial\phi_K^{(0,2)}|^2\\
\lesssim&M_0^2e^{2M_1}U^\delta\int_{\rD_{0,U}^{0,V}} r^{-1} v|\spartial\phi_K^{(0,2)}|^2\\
\lesssim&M_0^4e^{4M_1}U^\delta V^{2\delta}.
\end{Eq*}
Finally for the lower order term $r^{-1}\phi_K$  in $I_2$, we have 
\begin{Eq*}
&\int_0^U\int_u^V u^{1-\frac{\delta}{2}}v^{1-\frac{\delta}{2}}\|\phi_K^{(1,2)}\|_{L_\omega^2}^2\| r^{-1}\phi_K^{(0,2)}\|_{L_\omega^2}^2 r^2\d v\d u\\
\lesssim&\int_0^U\int_u^V u^{1-\frac{\delta}{2}}v^{1-\frac{\delta}{2}}\|\phi_K^{(1,2)}\|_{L_\omega^2}^2(M_0^2e^{2M_1} r^{-2}v^{2\delta-2}) r^2\d v\d u\\
%\lesssim&M_0^2e^{M_1} \int_0^U u^{\delta}\int_u^V \int_{\rS^2}|\phi_K^{(1,2)}|^2 \d\omega\d v\d u\\
\lesssim&M_0^4e^{4M_1}U^\delta V^{2\delta}.
\end{Eq*}
The term $I_3$ is easier and can be controlled in a similar way
$$I_3 \lesssim M_0^6e^{6M_1}U^{5\delta}.$$ 
 Combining all the above estimates, we have shown that for all  $U\leq \min\{\varepsilon,V\}\leq 1$
 \begin{Eq*}
E[w_{K+1}^{(1,2)}](U,V)\leq&CU^{\frac{\delta}{2}}V^{\frac{\delta}{2}}(I_1+I_2+I_3)\\
\leq& CU^{\frac{\delta}{2}}V^{\frac{\delta}{2}}\kl(M_0^4e^{4M_1}U^\delta V^{2\delta}+\frac{M_0^4}{M_1}e^{2M_1V^{2\delta}}U^{\frac{\delta}{2}}V^{\frac{\delta}{2}}+M_0^6e^{6M_1}U^{5\delta}\kr)\\
\leq&\kl(CM_0^2e^{4M_1}\varepsilon^{\frac{\delta}{2}}+C\frac{M_0^2}{M_1}+CM_0^4e^{6M_1}\varepsilon^{4\delta}\kr)M_0^2e^{2M_1V^{2\delta}}U^\delta V^\delta.
\end{Eq*}
Now Let 
 $$M_1=4CM_0^2,\quad \varepsilon=(4CM_0^2e^{4M_1})^{-\frac{2}{\delta}}.$$ 
We then conclude that the linear solution $w_{k+1}$ also verifies the energy estimate \eqref{Eq:Eoeons}. We thus finished the proof for the Proposition.  
\end{proof}

\subsection{End of the proof for Theorem \Tm{Tm:M1}}
 To show the existence of the solution, we also need to show that the sequence of linear solutions  $\{w_k\}$ is Cauchy sequence in certain sense. 
First of all  similarly to the proof of \eqref{Eq:Eoeons}, we can show that 
\begin{Eq*}
&U^{-\delta} V^{-\delta}e^{-2M_1V^{2\delta}}E[w_{K+1}-w_{K}](U,V)\\
\leq& CU^{-\delta} V^{-\delta}e^{-2M_1V^{2\delta}}\int_{\rD_{0,U}^{0,V}}\kl|(v\hL+u\hLb)(w_{K+1}-w_K)\kr|\cdot|F_K-F_{K-1}|\\
\leq&CU^{-\frac{\delta}{2}}V^{-\frac{\delta}{2}}e^{-2M_1V^{2\delta}}\int_{\rD_{0,U}^{0,V}}u^{1-\frac{\delta}{2}}v^{1-\frac{\delta}{2}}|F_K-F_{K-1}|^2+ \frac{U^{-\frac{\delta}{2}}V^{-\frac{\delta}{2}}e^{-2M_1V^{2\delta}}}{2}E[w_{K+1}-w_K](U,V),
\end{Eq*}
in which the second term on right hand side can be absorbed. It then suffices to control the difference of the nonlinearity. By the assumption, notice that 
\begin{align*}
|F_K-F_{K-1}|\lesssim & |(\partial\phi_K,\partial\phi_{K-1})||w_K-w_{K-1}|+|(\phi_K,\phi_{K-1})||\partial\kl(w_K-w_{K-1}\kr)|\\ 
&+|(\phi_K,\phi_{K-1})|^2 |w_K-w_{K-1}|.
\end{align*}
The bound for these terms on the right hand is similar to that for $I_1$, $ I_2$, $I_3$. With out loss of generality, we only deal with the most difficult term  $|\phi_K||\hLb\kl(w_K-w_{K-1}\kr)|$. In fact in view of the estimates \eqref{Eq:Eolsos2} and \eqref{Eq:Eonsos2}, we can show that 
\begin{Eq*}
I_{1,2}':=&\int_{\rD_{0,U}^{0,V}}u^{1-\frac{\delta}{2}}v^{1-\frac{\delta}{2}}|\phi_K|^2|\hLb\kl(w_K-w_{K-1}\kr)|^2\\
\lesssim&\int_0^U\int_u^Vu^{1-\frac{\delta}{2}}v^{1-\frac{\delta}{2}}\|\phi_K^{(0,2)}\|_{L_\omega^2}^2\|\hLb(w_K-w_{K-1})\|_{L_\omega^2}^2 r^2\d v\d u\\
\lesssim&\int_0^U\int_u^Vu^{1-\frac{\delta}{2}}v^{1-\frac{\delta}{2}}(M_0^2v^{2\delta-2}+M_0^2e^{2M_1}u^\delta v^{\delta-2})\|\hLb(w_K-w_{K-1})\|_{L_\omega^2}^2 r^2\d v\d u\\
\lesssim&V^{\frac{\delta}{2}}\int_0^V (M_0^2v^{\delta-1}+M_0^2e^{2M_1}U^\delta v^{-1}) \int_{\rHb_{0,U}^{v}} u^{1-\frac{\delta}{2}}|\hLb(w_K-w_{K-1})|^2 \d v.
\end{Eq*}
Then by using  \La{La:kt4}, we estimate that  
\begin{Eq*}
I_{1,2}'\lesssim&V^{\frac{\delta}{2}}\int_0^V (M_0^2v^{\delta-1}+M_0^2e^{2M_1}U^\delta v^{-1}) \kl(U^{\frac{\delta}{2}}v^\delta e^{2M_1v^{2\delta}}\kr)\kl(U^{-\delta} v^{-\delta}e^{-2M_1v^{2\delta}}E[w_{K+1}-w_K](U,v)\kr) \d v\\
\lesssim&V^{\frac{\delta}{2}}\int_0^V M_0^2e^{2M_1v^{2\delta}}U^{\frac{\delta}{2}}v^{2\delta-1}+M_0^2e^{4M_1}U^\frac{3\delta}{2} v^{\delta-1}\d v\cdot \sup_v\kl(U^{-\delta} v^{-\delta}e^{-2M_1v^{2\delta}}E[w_{K+1}-w_K](U,v)\kr)\\
\lesssim&\kl(\frac{M_0^2}{M_1} e^{2M_1V^{2\delta}}U^{\frac{\delta}{2}}V^{\frac{\delta}{2}}+M_0^2e^{4M_1}U^{\frac{3\delta}{2}}V^{\frac{3\delta}{2}}\kr)\cdot \sup_v\kl(U^{-\delta} v^{-\delta}e^{-2M_1v^{2\delta}}E[w_{K+1}-w_K](U,v)\kr).
\end{Eq*}
Combining these estimate, we then obtain that 
\begin{Eq*}
&U^{-\delta} V^{-\delta}e^{-2M_1V^{2\delta}}E[w_{K+1}-w_{K}](U,V)\\
\leq&C\kl(M_0^2e^{4M_1-2M_1V^{2\delta}}U^{\frac{\delta}{2}} V^{\frac{3\delta}{2}}+\frac{M_0^2}{M_1}+M_0^4e^{6M_1-2M_1V^{2\delta}}U^{4\delta}\kr)\cdot \sup_v\kl(U^{-\delta} v^{-\delta}e^{-2M_1v^{2\delta}}E[w_{K+1}-w_K](U,v)\kr)\\
\leq&\kl(CM_0^2e^{4M_1}\varepsilon^{\frac{\delta}{2}}+\frac{CM_0^2}{M_1}+CM_0^4e^{6M_1}\varepsilon^{4\delta}\kr)\cdot \sup_v\kl(U^{-\delta} v^{-\delta}e^{-2M_1v^{2\delta}}E[w_{K+1}-w_K](U,v)\kr).
\end{Eq*}
In particular by taking 
 $$M_1=4CM_0^2,\quad \varepsilon=(4CM_0^2e^{4M_1})^{-\frac{2}{\delta}},$$ 
we see that $w_{K}$ is Cauchy sequence in the space with the weighted energy norm $E[\cdot]$.  Hence in view of \Pn{La:Eoi}, we conclude that $w_k$ converges to some solution $w$, such that $\phi:=\phi_0+w$ which  solves the equation \eqref{Eq:Me}.
It is also obvious that the solution is unique in the sense of  \eqref{Eq:Er}.
Now we thus finished the proof of \Tm{Tm:M1}.

\section{Proof of \Tm{Tm:M2}}
As an application of \Tm{Tm:M1}, we now show that the regularity for the scattering problem for the Maxwell-Klein-Gordon equation studied in our previous work  \cite{Yang:MKG:scattering} with data prescribed on the future null infinity could be improved. The future null infinity is topologically equivalent to $\mathbb{R}\times \mathbb{S}^2$, which has two ends with retard time $u\in \mathbb{R}$ tending to $\pm \infty$. The data on the future null infinity is localized in the sense that it decays in terms of $u$. By using our local existence result developed here with singular data, we could improve the decay rate assumption for the data near the time like infinity when $u\rightarrow \infty$.  The proof is similar to that in \cite{Yang:MKG:scattering}. For simplicity, we will only sketch the proof here.  

\subsection{Conformal compactification}
For the scattering problem for the Maxwell-Klein-Gordon system studied in \cite{Yang:MKG:scattering}, it can be viewed as characteristic initial value problem for the nonlinear system with data prescribed at the future null infinity. The future null infinity of the Minkowski space $\mathbb{R}^{1+3}$ is an infinite conic null cone. Hence the first step to construct the solution for the full system is to show the existence of the solution in a neighborhood of the time like infinity which is the conic point of the null cone. This was carried out via the conformal compactification method. 

For fixed $U_*\in \mathbb{R}$, define
\begin{Eq*}
&T_*:=2U_*-2^{-1},\qquad \Lambda:=(t-T_*)^2-|x|^2,\qquad (\tt,\tx)=\Lambda^{-1}(t-T_*,x),\\
&\tphi(\tt,\tx):=\Lambda\phi(t,x),\qquad\tA(\tt,\tx):=A(t,x),\qquad \tF:=\d\tA,
\end{Eq*}
The conformal symmetry of the Maxwell-Klein-Gordon system indicates that 
 $(F,\phi)$ solves the system \eqref{Eq:MKGe-Fphi} in the region $\rD_{U_*,\infty}^{U_*,\infty}$ is equivalent to  that  
$(\tF,\tphi)$ solves the same system in the compact region $\trD_{0,1}^{0,1}$ under $(\tt,\tx)$ coordinates. The initial data on the future null infinity then give the characteristic data on the conic cone $\tilde{\mathcal{H}}_0$. 

The tilde quantities  are the associated ones under the coordinates $(\tt,\tx)$, which are defined in the same way as those under the coordinates  $(t,x)$.
By defining $$u_*:=2^{-1}(t-T_*-r),\quad v_*:=2^{-1}(t-T_*+r),$$
 we can compute that  
\begin{Eq}\label{Eq:Rbofatf}%Relation between origin frame and transformed frame
&\tu=(4v_*)^{-1},&\quad &\tv=(4u_*)^{-1}, &\quad& \Lambda=4u_*v_*=\tr^{-1}r,\\
&\tLb=-4v_*^2L,&\quad& \tL=-4u_*^2\Lb,&\quad& \te_i=4u_*v_* e_i,&\quad&\tomega=\omega.
\end{Eq}
In particular the null components of the Maxwell field and the scalar field verify the relation
\begin{Eq*}
\talpha_A(0,\tv,\tomega)=&\lim_{\tu\rightarrow 0}\tF_{\tL\te_A}(\tu,\tv,\tomega)=-16u_*^3\lim_{v\rightarrow\infty}[v_*F_{\Lb e_A}](u,v,\omega)=-16u_*^3\Alphab_A(u,\omega),\\
\tphi(0,\tv,\tomega)=&\tv^{-1}\lim_{\tu\rightarrow 0}[\tr\tphi](\tu,\tv,\tomega)=4u_*\lim_{v\rightarrow \infty}[r\phi](u,v,\omega)=4u_*\Phi(u,\omega).
\end{Eq*}
In view of the assumption that $\|(\Alphab,\Phi)\|_{S\!N_{U_*}}<\infty$, we then conclude that 
\begin{Eq}\label{Eq:Eo_talpha_o_tH}%Estimate of talpha on tH
&\sum\Low{n_1\leq 1\\n_2\leq6-3n_1}\int_0^1\int_{\trS^2}\tv^{3-2\delta}\kl|(\tv\tL)^{n_1}\tpartial_{\tomega}^{n_2} \talpha\kr|_{\tu=0}^2\d\tomega\d \tv\\
\lesssim&\sum\Low{n_1\leq 1\\n_2\leq6-3n_1}\int_{U_*}^\infty \int_{\rS^2} u_*^{-3+2\delta}\kl|(u_*\Lb)^{n_1}\partial_\omega^{n_2} (u_*^3\Alphab)\kr|^2\d\omega\frac{\d u}{u_*^2}\\
\lesssim&\sum\Low{n_1\leq 1\\n_2\leq6-3n_1}\int_{U_*}^\infty \int_{\rS^2} \kl<u\kr>^{1+2\delta+2n_1}\kl|\Lb^{n_1}\partial_\omega^{n_2} \Alphab\kr|^2\d\omega\d u<\infty
\end{Eq}
and similarly
\begin{Eq}\label{Eq:Eo_tsAphi_o_tH}
&\sum\Low{n_1\leq 2\\n_2\leq 5-2\max\{n_1-1,0\}}\int_0^1\int_{\trS^2}\tv^{1-2\delta}\kl|(\tv\tD_\tL)^{n_1}\tD_\tomega^{n_2}\tphi\kr|_{\tu=0}^2\d\tomega\d \tv%\\
%\lesssim&\sum\Low{n_1\leq 2\\n_2\leq 5-n_1}\int_0^1\int_{\trS^2}\tv^{-1-2\delta}\kl|(\tv\tD_\tL)^{n_1}\tD_\tomega^{n_2}(\tr\tphi)\kr|^2\d\tomega\d \tv\\
%\lesssim&\sum\Low{n_1\leq 2\\n_2\leq 5-n_1}\int_{U_*}^\infty\int_{\trS^2}u_*^{1+2\delta}\kl|(u_*D_\Lb)^{n_1}D_\omega^{n_2}\Phi\kr|^2\d\omega\frac{\d u}{u_*^2}\\
<\infty.
\end{Eq}

\subsection{Reducing to the characteristic initial value problem}
After the above comformal transformation, we are now reducing the local existence of solution to the original Maxwell-Klein-Gordon system in a neighborhood of the timelike infinity to the characteristic initial value problem to the same system on the conic region $\rD_{0,1}^{0,1}$. For simplicity, in this section, we omit the tilde symbol. To construct the local solution to the equation \eqref{Eq:MKGe-Fphi}, we make use of the \emph{Lorentz} gauge
\[
\pa^\mu A_\mu =0,
\]
under which the equation \eqref{Eq:MKGe-Fphi} is reduced to the nonlinear wave equaiton  for  $\phi$ and $A$
\begin{Eq*}
\begin{cases}
\square A_\mu=-\Im(\phi\cdot\overline{\partial_\mu \phi})+A_\mu|\phi|^2,\\
\square \phi=-2\I A^\mu\partial_\mu\phi+A^\mu A_\mu\phi.
\end{cases}
\end{Eq*}
This is of the type of equation \eqref{Eq:Me}. To apply \Tm{Tm:M1}, we then need to show that the initial data for $A$ and $\phi$ verify the assumption \eqref{Eq:Cod}. First of all, notice that the Lorentz gauge condition is equivalent to find some function $\chi$ such that 
\[
\Box\chi=\pa^\mu A_\mu
\] 
for some given connection field $A$. For the conic region $\rD_{0,1}^{0,1}$ we are interested in here, it then suffices to prescribe the data for $\chi$ on the cone $\mathcal{H}_0^{0, 1}$, which can be realized by simply setting 
\[
\chi(0)=0,\quad L\chi= A_L.
\]
In particular such special data for $\chi$ leads to the conclusion that $A_L=0$ on $\rH_{0}^{0,1}$. Then by using the equation, we can compute the other components for the connection field $A$ on $\rH_{0}^{0,1}$
\begin{Eq}\label{Eq:c_eq_ALb}
L(rL(rA_\Lb))=\sum\Low{n_1\leq 1, n_2\leq 1}C^j_{n_1,n_2}(rL)^{n_1}\partial_\omega^{n_2}A_{e_j}-r^2\Im(\phi\cdot \overline{L\phi}),\qquad
L(rA_{e_j})=r\alpha_j.
\end{Eq}
See Section 6 in \cite{Yang:MKG:scattering} for detailed discussion.
Since we are restricting the computations on the initial hypersurface $\mathcal{H}_0^{0, 1}$, we will omit the lower-script $u=0$ for simplicity.

For the angular part $\sA:=(A_{e_1},A_{e_2})$, first we show that 
\begin{Eq*}
E_0[\sA]:=&\sum\Low{n_1\leq 2\\n_2\leq 6-3\max\{0,n_1-1\}}\int_0^1\int_{\rS^2}r^{1-2\delta}|(rL)^{n_1}\partial_\omega^{n_2}\sA|^2\d\omega\d r\\
\lesssim&\Bigg(\sum\Low{n_1=0\\n_2\leq 6}+\sum\Low{1\leq n_1\leq 2\\n_2\leq 9-3n_1}\Bigg)\int_0^1\int_{\rS^2}r^{-1-2\delta}|(rL)^{n_1}\partial_\omega^{n_2}(r\sA)|^2\d\omega\d r.
\end{Eq*}
We use the inequality \eqref{Eq:kt3} to bound  the first summation term on the right hand side. For the second term, note that $rL$ commutes with $\partial_\omega$. In view of  \eqref{Eq:Eo_talpha_o_tH}, we obtain that 
\begin{Eq*}
E_0[\sA]\lesssim&\sum\Low{n_1\leq 1\\n_2\leq 6-3n_1}\int_0^1\int_{\rS^2}r^{1-2\delta}|(rL)^{n_1}\partial_\omega^{n_2}L(r\sA)|^2\d\omega\d r\\
\lesssim&\sum\Low{n_1\leq 1\\n_2\leq 6-3n_1}\int_0^1\int_{\rS^2}r^{3-2\delta}|(rL)^{n_1}\partial_\omega^{n_2}\alpha|^2\d\omega\d r\\
<&\infty.
\end{Eq*}
Moreover by using  the inequality \eqref{Eq:kt3} again, we also have  the bound  
\begin{Eq*}
H_0[\sA]:=&\sup_{r\in(0,1]}\sum\Low{n_1\leq 1\\n_2\leq 6-3n_1}r^{2-2\delta}\int_{\rS^2}|(rL)^{n_1}\partial_\omega^{n_2}\sA|^2\d\omega\\
\lesssim& \sup_{r\in(0,1]}\sum\Low{n_1\leq 1\\n_2\leq 6-3n_1}r^{-2\delta}\int_{\rS^2}|(rL)^{n_1}\partial_\omega^{n_2}(r\sA)|^2\d\omega\\
\lesssim& E_0[\sA]\\
<&\infty.
\end{Eq*}
Next for the scalar field $\phi$, by using \eqref{Eq:Eo_tsAphi_o_tH} and \eqref{Eq:kt3}, by definition, we have 
\begin{Eq*}
H_0^D[\phi]:=&\sup_{r\in(0,1]}\sum\Low{n_1\leq 1\\n_2\leq 5-2n_1}r^{2-2\delta}\int_{\rS^2}|(rD_L)^{n_1}D_\omega^{n_2}\phi|^2\d\omega\\
\lesssim&\sum\Low{n_1\leq 2\\n_2\leq 5-2\max\{n_1-1,0\}}\int_0^1\int_{\rS^2}r^{1-2\delta}|(rD_L)^{n_1}D_\omega^{n_2}\phi|^2\d\omega\d r\\
\equiv &E_0^D[\phi]\\
<&\infty.
\end{Eq*}
In particular H\"older's inequality together with Sobolev inequality on the sphere  leads to 
\begin{Eq*}
E_0[\phi]:=&\sum\Low{n_1\leq 2\\n_2\leq 5-2\max\{n_1-1,0\}}\int_0^1\int_{\rS^2}r^{1-2\delta}|(rL)^{n_1}\partial_\omega^{n_2}\phi|^2\d\omega\d r\\
\lesssim& \sum\Low{n_1\leq 2\\n_2\leq 5-2\max\{n_1-1,0\}}\int_0^1\int_{\rS^2}r^{1-2\delta}|(rD_L)^{n_1}(D_\omega-ir\sA)^{n_2}\phi|^2\d\omega\d r\\
\lesssim&\kl(E_0[\sA]+E_0^D[\phi]\kr)\kl(1+H_0[\sA]+H_0^D[\phi]\kr)^4\\
<&\infty.
\end{Eq*}
By a similar argument, we also have 
\begin{Eq*}
H_0[\phi]:=&\sup_{r\in(0,1]}\sum\Low{n_1\leq 1\\n_2\leq 5-2n_1}r^{2-2\delta}\int_{\rS^2}|(rL)^{n_1}\partial_\omega^{n_2}\phi|^2\d\omega\\
\lesssim& E_0[\phi]\\
<&\infty.
\end{Eq*}
Once we have the desired estimates for the angular part of the connection field and the scalar field, we are now ready to control the tangential part $A_\Lb$ of the connection field.  In view of the inequality  \eqref{Eq:kt3}  and the estimate \eqref{Eq:c_eq_ALb}, we can show that 
\begin{Eq*}
E_0[A_\Lb]:=&\sum\Low{n_1\leq 2\\n_2\leq 5-n_1}\int_0^1\int_{\rS^2}r^{1-2\delta}|(rL)^{n_1}\partial_\omega^{n_2}A_\Lb|^2\d\omega\d r\\
\lesssim&\Bigg(\sum\Low{n_1=0\\n_2\leq 5}+\sum\Low{1\leq n_1\leq 2\\n_2\leq 5-n_1}\Bigg)\int_0^1\int_{\rS^2}r^{-1-2\delta}|(rL)^{n_1}\partial_\omega^{n_2}(rA_\Lb)|^2\d\omega\d r\\
\lesssim&\sum\Low{n_1\leq 1\\n_2\leq 5}\int_0^1\int_{\rS^2}r^{1-2\delta}|(rL)^{n_1}\partial_\omega^{n_2}L(rA_\Lb)|^2\d\omega\d r\\
\lesssim&\sum\Low{n_2\leq 5}\int_0^1\int_{\rS^2}r^{1-2\delta}|\partial_\omega^{n_2}L(rL(rA_\Lb))|^2\d\omega\d r\\
\lesssim&\sum\Low{n_2\leq 5}\int_0^1\int_{\rS^2}r^{1-2\delta}\kl|\partial_\omega^{n_2}\kl((rL)^{\leq 1}\partial_\omega^{\leq 1}\sA,r^2\phi\cdot\overline{L\phi}\kr)\kr|^2\d\omega\d r\\
\lesssim&E_0[\sA]+E_0[\phi]H_0[\phi]\\
<&\infty.
\end{Eq*}
Therefor we have shown that under the assumption of Theorem \ref{Tm:M2} and the Lorentz gauge condition, the initial data for the connection field $A$ and the scalar field $\phi$ verify the conditions in Theorem \ref{Tm:M1}. 

\subsection{Proof for Theorem \ref{Tm:M2}}
Based on the above argument, by using Theorem \ref{Tm:M1}, we conclude that there exists a solution $(\tilde{A}, \tilde{\phi})$ to the Maxwell-Klein-Gordon equation under Lorenz gauge 
in a neighborhood of the initial cone
 $$\{\tt/(\tt^2-\tr^2)\geq (2\varepsilon)^{-1},~\tv\leq 1\}\subset\trD_{0,\varepsilon}^{0,1}$$
  for some small constant $\varepsilon>0$. In particular the  weighted energy of $(\tF,\tphi)$ on the hyperboloid 
$$\tilde\varC:=\{\tt/(\tt^2-\tr^2)= (2\varepsilon)^{-1},~\tv\leq 1\}$$ 
is finite. Since this hyperboloid is away from the conic point, by doing the reverse conformal transformation onto the original Minkowski space, we can conclude that the solution $(F,\phi)$ to the Maxwell-Klein-Gordon equation in the region 
$$\{t\geq(2\varepsilon)^{-1}+T_*,~u\geq U_*\}\subset \rD_{U_*,\infty}^{(4\varepsilon)^{-1}+2^{-1}T_*,\infty}$$
has finite energy  on the Cauchy hypersurface $\{t=(2\varepsilon)^{-1}+T_*,~u\geq U_*\}$. Then in view of the classical result of Klainerman-Machedon in \cite{MKGkl}, we  can show that the solution exists   in the region 
$\{t\leq(2\delta)^{-1}+T_*,~u\geq U_*\}$.
We thus have extended the solution to the whole region  $\rD_{U_*,\infty}^{U_*,\infty}$ and boundedness of the weighted energy norms follows accordingly via conformal transformation. We thus finished the proof for  Theorem \ref{Tm:M2}.

%\subsection*{Acknowledgment}
%AcknowledgmentAcknowledgmentAcknowledgmentAcknowledgmentAcknowledgmentAcknowledgmentAcknowledgmentAcknowledgment

%\bibliography{shiwu.bib}

\end{document}